\documentclass{degruyter-proceedings}          

\title{Rational integrability of trigonometric polynomial potentials on the flat torus}
\headlinetitle{Integrability of trigonometric polynomial potentials}

\lastnameone{Combot}
\firstnameone{Thierry}
\nameshortone{T. Combot}
\addressone{Piazza Cavalieri, 56127 Pisa}
\countryone{Italy}
\emailone{thierry.combot@u-bourgogne.fr}

\lastnametwo{}
\firstnametwo{}
\nameshorttwo{}
\addresstwo{}
\countrytwo{}
\emailtwo{}

\newtheorem{defi}{Definition}
\newtheorem{thm}{Theorem}
\newtheorem{prop}{Proposition}
\newtheorem{cor}{Corollary}
\newtheorem{lem}{Lemma}

\lastnamethree{}
\firstnamethree{}
\nameshortthree{}
\addressthree{}
\countrythree{}
\emailthree{}

\lastnamefour{}
\firstnamefour{}
\nameshortfour{}
\addressfour{}
\countryfour{}
\emailfour{}

\lastnamefive{}
\firstnamefive{}
\nameshortfive{}
\addressfive{}
\countryfive{}
\emailfive{}

\abstract{We consider a lattice $\mathcal{L}\subset \mathbb{R}^n$ and a trigonometric potential $V$ with frequencies $k\in\mathcal{L}$. We then prove a strong integrability condition on $V$, using the support of its Fourrier transform. We then use this condition to prove that a real trigonometric polynomial potential is integrable if and only if it separates up to rotation of the coordinates. Removing the real condition, we also make a classification of integrable potentials in dimension $2$ and $3$, and recover several integrable cases. These potentials after a complex variable change become real, and correspond to generalized Toda integrable potentials. Moreover, along the proof, some of them with high degree first integrals are explicitly integrated.}



\begin{document}


\section{Integrability of potentials on a flat torus}

In this article, we consider trigonometric polynomial potentials
$$V(q)=\sum\limits_{k\in\mathcal{L}^*} a_ke^{ik.q}$$
with $\mathcal{L}\subset \mathbb{R}^n$ a $n$-dimensional lattice and $a_k\in\mathbb{C}$ with finitely many non zero $a_k$. Remark that the coefficient $a_0$ is assumed to be $0$ as a potential is defined up to addition of a constant. This defines a multiperiodic function $V$, with periods the lattice $2\pi\mathcal{L}$. Thus after quotient, the function $V$ can be seen as a function on $\mathcal{T}_n=\mathbb{R}^n/(2\pi\mathcal{L})$. The function $V$ has a Fourier transform
$$\hat{V}:\mathcal{L} \mapsto \mathbb{C},\quad \hat{V}(k)=a_k$$
This type of potential is for the torus the equivalent of polynomial potentials on $\mathbb{R}^n$. The periodicity condition simply requires to replace monomials in $q$ by periodic functions, and the exponentials are the simplest ones. As for polynomial potentials on $\mathbb{R}^n$, the coefficients $a_k$ need not to be real, nor the values of $V$. In all generality, the potential $V$ has thus complex values, and as a trigonometric polynomial can be extended on the complex domain $\mathbb{C}^n/(2\pi\mathcal{L})$. Still, two interesting real cases appear.
\begin{itemize}
\item The case of $V$ with real values on $\mathcal{T}_n$ can be written using complex conjugacy condition $\overline{V}=V$, giving the condition on the coefficients $\bar{a}_k=a_{-k}$. We will then say that $V$ is a real trigonometric polynomial potential.
\item The case with all coefficients $a_k$ real leads to a potential $V$ with real values on $i\mathbb{R}^n/(2\pi\mathcal{L})$. After multiplying by $i$ the coordinates, we then obtain a real function
$$V(q)=\sum\limits_{k\in\mathcal{L}} a_ke^{k.q}$$
which generalize Toda type potentials \cite{1,2,3,4,17,18}. This potential is still multiperiodic, but with purely imaginary periods in the lattice $2i\pi\mathcal{L}$.
\end{itemize}

A potential $V$ is associated to a Hamiltonian system with Hamiltonian
$$H=\frac{1}{2}\sum\limits_{i=1}^n p_i^2+V(q)$$
This defines a Hamiltonian system with $n$ degrees of freedom, for which a good definition of integrability exists.

\begin{defi}
A Hamiltonian $H$ with $n$ degrees of freedom is called Liouville integrable if there exists $n$ functions $F_1=H,\dots,F_n$ which pairwise Poisson commute and are functionally independent.
\end{defi}

An important point in the search of first integrals, and even more in non integrability proofs is to precise the class of functions in which the first integrals are searched. Let us note $\omega_1,\dots,\omega_n$ a basis of the lattice $\mathcal{L}$. The Hamiltonian is a rational function in $(p_j,\exp{i\omega_j.q})_{j=1\dots n}$. The set of such rational functions
$$K=\mathbb{C}\left(p_1,\dots,p_n,e^{i\omega_1.q},\dots, e^{i\omega_n.q}\right)$$
is a differential field (a field which is stable by the derivations $\partial_{p_i},\partial_{q_1}$), and thus is a reasonable class for searching first integrals. Thus in the following, we will always search for first integrals in $K$, and prove non integrability for first integrals in $K$. We will obtain the following main result in section $2$.

\begin{thm}\label{thmmain}
Let $V$ be a Liouville integrable trigonometric polynomial potential on $\mathcal{T}_n$. Let us note $\mathcal{C}$ the convex hull of the support $\mathcal{S}$ of $\hat{V}$, and $v_1,\dots,v_p$ its summits. Then
\begin{itemize}
\item The set $\mathcal{S}$ is included in $\bigcup_{i=1}^p \mathbb{R}^+.v_i$.
\item If $v_i,v_j$ belong to the same edge, the angle $\widehat{v_iv_j}\in\left\lbrace 0,\frac{\pi}2,\frac{2\pi}3,\frac{3\pi}{4},\frac{5\pi}{6},\pi\right\rbrace$.
\item If $v_i,v_j$ belong to the same edge and $\widehat{v_iv_j}=\frac{2\pi}{3},\frac{3\pi}{4},\frac{5\pi}{6}$ then $\lVert v_i\lVert/\lVert v_j\lVert=1,2^{\pm 1/2},$ $3^{\pm 1/2}$ respectively.
\item If $v_i,v_j$ belong to the same edge and $\widehat{v_iv_j}\in\{\frac{2\pi}{3},\frac{5\pi}{6}\}$ then
$$\mathcal{S}\cap \mathbb{R}_+.v_i=\{v_i\},\quad \mathcal{S}\cap \mathbb{R}_+.v_j=\{v_j\}.$$
\item If $v_i,v_j$ belong to the same edge, $\widehat{v_iv_j}=\frac{3\pi}{4}$ and $\lVert v_i\lVert/\lVert v_j\lVert=\sqrt{2}$ then
$$\mathcal{S}\cap \mathbb{R}_+.v_i\subset\{v_i,v_i/2\},\quad \mathcal{S}\cap \mathbb{R}_+.v_j=\{v_j\}.$$
\end{itemize}
\end{thm}

Theorem \ref{thmmain} implies that the support of the Fourier transform of an integrable potential cannot be large, as it should be included in finitely many semi-straight lines with specific (and quite large) angle between them. These conditions are very similar to \cite{17,16}, which hold in the case of polynomial first integrals, whereas Theorem \ref{thmmain} holds for rational first integrals. A similar result in given in \cite{20} assuming algebraic integrability (and so rational first integrals), but in the case where $\mathcal{S}$ has $n+1$ independent elements. This would imply here that in particular, $\mathcal{S}$ only has points in its convex envelope, and thus the statements of Theorem \ref{thmmain} about the interior of $\mathcal{C}$ are not covered. Theorem \ref{thmmain} will allow us to prove the following two corollaries

\begin{cor}\label{cor1}
A real trigonometric polynomial potential $V$ is Liouville integrable if and only if up to rotation of the coordinates, it can be written $V(q)=\sum_{i=1}^n V_i(q_i)$ (called a separable potential).
\end{cor}

Related results exist when the trigonometric polynomial condition is dropped but with additional conditions on the degree of first integrals \cite{11,12,13,15}. This corollary also proves in the case of trigonometric polynomial potential the conjecture that any additional irreducible first integral of an integrable case have a degree in momenta at most $2$, which is related to a conjecture of Kozlov \cite{14}. For complex potentials $V$, much more integrable cases are possible. We make a classification of them in dimension $2,3$.

\begin{cor}\label{cor2}
A Liouville integrable trigonometric polynomial potential in dimension $2$ or $3$ belongs to one of the following cases
\begin{itemize}
\item The separable potentials.
\item In dimension $2$, up to rotation/dilatation/symmetry of the coordinates
\begin{align*}
V(q_1,q_2)=& ae^{2iq_1}+be^{-iq_1+i\sqrt{3}q_2}+ce^{-iq_1-i\sqrt{3}q_2}\\
V(q_1,q_2)=& ae^{2iq_1}+be^{2iq_2}+ce^{-2iq_1-2iq_2}+de^{-iq_1-iq_2}\\
V(q_1,q_2)=& ae^{2iq_1}+be^{iq_1}+ce^{2iq_2}+de^{iq_2}+ee^{-iq_1-iq_2}\\
V(q_1,q_2)=& ae^{2iq_1}+be^{2i\sqrt{3}q_2}+ce^{-iq_1-i\sqrt{3}q_2}\\
V(q_1,q_2)=& ae^{2i\sqrt{3}q_1}+be^{2iq_2}+ce^{-i\sqrt{3}q_1-3iq_2}
\end{align*}
with $a,b,c\in\mathbb{C}$.
\item In dimension $3$, partially separable potentials, i.e. up to rotation of the coordinates which can be written $V(q)=V_1(q_1,q_2)+V_3(q_3)$ with $V_1$ integrable.
\item In dimension $3$, up to rotation/dilatation/symmetry of the coordinates
\begin{align*}
V(q_1,q_2,q_3)=& ae^{iq_1+iq_2}+be^{-iq_2+iq_3}+ce^{-iq_2-iq_3}+de^{-iq_1+iq_2}\\
V(q_1,q_2,q_3)=& ae^{2iq_1}+be^{-iq_1+iq_2}+ce^{-iq_2-iq_3}+de^{-iq_2+iq_3}+ee^{iq_1}\\
V(q_1,q_2,q_3)=& ae^{2iq_1}+be^{-iq_1+iq_2}+ce^{-iq_2+iq_3}+de^{-2iq_3}+ee^{iq_1}+fe^{-iq_3}\\
\end{align*}
with $a,b,c,d,e,f\in\mathbb{C}$.
\end{itemize}
\end{cor}




These non separable potentials are integrable but with high degree first integrals (degree more than $2$ and up to $6$ in $p_i$). Most of cases in dimension $2$ were already found by Hietarinta \cite{5} and the third in \cite{18}. The first non partially separable case in dimension $3$ is related to an integrable Toda potential \cite{1}, the other ones could be recovered through building convenient Dynkin diagrams following \cite{17}, with the third case found in Ranada \cite{19}. Remark that none of the non separable integrable potential are real, as the condition $\bar{a}_k=a_{-k}$ would imply $V=0$. However if the coefficients $a,b,c,d$ are chosen real, multiplying all the coordinates by $i$ produce real valued potential, which is a finite sum of real exponentials. In the computation, the potential
\begin{equation}\label{eqpotnonint}
V(q_1,q_2)= ae^{6iq_1}+be^{4iq_1}+ce^{2iq_1}+de^{2i\sqrt{3}q_1}+ee^{-3iq_1-\sqrt{3}iq_2}
\end{equation}
although proved to be non integrable, distinguished itself by passing all integrability tests of Theorem \ref{thmmain} except condition $4$, and the perturbation analysis in section 2.6 made appear it more regular than the others. This suggests that it could admit a first integral not rational but with simple ramification structure, as a Darbouxian first integral.\\

The corollaries are proven analysing all the possible configurations of the semi-straight lines generated by the $v_i$'s. In Corollary \ref{cor1}, the real assumption implies that the support of the Fourier transform of $V$ is invariant by central symmetry, and thus so is $\mathcal{C}$ and the summits $v_i$. This allows to tremendously reduce the possibilities, allowing to complete the proof in any dimension, see section $3$. In corollary \ref{cor2}, all possible configurations have to be analysed. This comes down to study polygonal tessellations of the sphere with particular edge length, and then reduces to a combinatorial problem. The number of possible combinations seems to grow exponentially with the dimension, and this is why we do not to look in dimension higher than $3$, although this could be done in principle. This analysis is done in section $4$. Once the possible sets of $v_i$'s are found, to conclude the proof of Corollary \ref{cor2}, we finally have to exhibit the first integrals of the resulting potentials, which is done in section $5$.

\section{Integrability conditions}

\subsection{Translation limits}

Let us first remark that the integrability status of a potential $V$ is invariant under some transformations: rotation, dilatation, translation, symmetry, multiplication of $V$ by a constant (which corresponds to a time change). This implies in particular that given an integrable potential, we can compute a whole family of potentials with several parameters. The integrable potentials are thus commonly classified up to these transformations. However, here we want to use these transformations to define limit potentials

\begin{defi}
Let $V$ be a trigonometric polynomial potential, and $v\in\mathbb{R}^n_*$ a non zero vector. The limit of the potential $V$ along the vector $v$ is defined the first non zero term of the series $V(q-i\ln \alpha v)$ at $\alpha=\infty$.
\end{defi}

\begin{prop}\label{proplimit}
Let $V$ be a trigonometric polynomial potential, $v\in\mathbb{R}^n_*$ a non zero vector, $P_\alpha$ the normal plane to $v$ with $\alpha v \in P_\alpha$ and
$$\alpha_0=\max\limits_{\alpha\in\mathbb{R}} \{\alpha,\; P_\alpha \cap \mathcal{S}(V)\neq \emptyset\}.$$
Then
$$V_v=\sum\limits_{k\in\mathcal{L}^*\cap P_{\alpha_0}} a_ke^{ik.q}$$
\end{prop}

\begin{proof}
Let us write down 
\begin{align*}
V(q-i\ln \alpha v) & =  \sum\limits_{k\in\mathcal{L}^*} a_ke^{ik.(q-i\ln\alpha v)}\\
 & =  \sum\limits_{k\in\mathcal{L}^*} a_k\alpha^{k.v}e^{ik.q}
\end{align*}
So the dominant terms in the series are those with maximal $\alpha^{k.v}$ with $a_k\neq 0$. The support $\mathcal{S}(V)$ is exactly the set of non zero $a_k$. The plane $P_\alpha$ contains all vectors $k$ such that $k.v$ equal to a specific value. And the $\alpha_0$ has been chosen such that $P_{\alpha_0} \cap \mathcal{S}(V)$ is non empty, with $\alpha_0$ as large as possible. Thus $P_{\alpha_0} \cap \mathcal{S}(V)$ exactly contains the $k$ maximizing the scalar product $k.v$ and non empty. This gives us the Proposition.
\end{proof}

\begin{prop}\label{propint}
If $V$ is integrable, then so is $V_v$.
\end{prop}

\begin{proof}
Let us note $e_j=\exp(i\omega_j.q),\;j=1\dots n$ where $\omega_1,\dots,\omega_n$ is a basis of $\mathcal{L}$. Now the Hamiltonian is a rational function in $p_j,e_j$. The transformation
$$\varphi_\alpha:q\rightarrow q-i\ln\alpha v$$
acts on the $e_j$ by
$$e_j \rightarrow \alpha^{\omega_j.v} e_j$$
The limit potential $V_v$ is the weight homogeneous part of maximal degree of $V(\varphi_\alpha(q))$. Let us note $m$ the total homogeneity degree of $V_v$. The Hamiltonian
$$\frac{1}{2}\sum\limits_{i=1}^n p_i^2+V_v(q)$$
is simply the weight homogeneous part of maximal degree of $H(\alpha^{m/2}p,\varphi_\alpha(q))$, which is rational in $p_i,e_i$.

Let us now consider the additional independent commuting first integrals of $V$, noting them $F_2,\dots,F_n$, and noting $F_1=H$. These functions belong to $K$ by assumption, and so are rational in $p_i,e_i$. Let us note $F_j^0$ the weight homogeneous part of maximal degree of $F_j(\alpha^{m/2}p,\varphi_\alpha(q))$. These are rational in $p_i,e_i$. As the $F_j^0$ are also the first non zero term in the series expansion of $F_j(\alpha^{m/2}p,\varphi_\alpha(q))$ at $\alpha=\infty$, we still have $\{F_i^0,F_j^0\}=0,\; \forall i,j$. However the independence property is no longer guaranteed.

We now use the Ziglin Lemma \cite{6}. If $F_1^0,\dots,F_{j-1}^0$ are functionally independent, then there exists a polynomial in $F_1,\dots,F_j$ such that its weight homogeneous part of maximal degree is independent with $F_1^0,\dots,F_{j-1}^0$. This Lemma can be applied recursively on the first integrals $F_1,\dots,F_n$, ensuring us to obtain $F_1^0\dots,F_n^0$ commuting first integrals with $F_1^0=\frac{1}{2}\sum\limits_{i=1}^n p_i^2+V_v(q)$ and functionally independent. This implies integrability of $V_v$.
\end{proof}

The Proposition \ref{propint} will be useful for us in the opposite way. We will try to prove that a limit $V_v$ for a suitable $v$ is not integrable, thus proving that $V$ is not integrable. From now on, we will note $\mathcal{S}(V)$ the support of the Fourrier transform $\hat{V}$ of $V$, and $\mathcal{C}(V)$ its convex hull.

Using these limit processes, we can build many integrable limit potential from one integrable potential. The convex set $\mathcal{C}(V)$ is always a polyhedron in $\mathbb{R}^n$, as it is the convex hull of a finite subset of $\mathcal{L}$. By choosing well the vector $v$, we can produce a limit potential $V_v$ containing exactly the terms $a_k$ with $k$ belonging to a summit, an edge, a face or a $p$-th dimensional face. We will now study integrability of these limit potentials.

\subsection{Angular conditions between summits}

Let us begin the the following Lemma

\begin{lem}\label{lemint}
Let us consider an integrable trigonometric polynomial potential $V$, with Hamiltonian $H$ and some function $J\in K$. Let us consider a generic orbit $\Gamma$. If the equation $\{I,H\}=J$ admits a solution in $K$, then
$$\int J(\Gamma(t)) dt \in K_{\mid_\Gamma}$$
where $K_{\mid_\Gamma}$ the differential field of functions in $K$ restricted on $\Gamma$.
\end{lem}

\begin{proof}
If $I$ is in $K$, then evaluating it on $\Gamma$ gives $I(\Gamma(t))\in K_{\mid_\Gamma}$ if $\Gamma$ is not in the pole set $\Sigma$ of $I$. This is generically the case, as this pole set is a codimension $1$ algebraic set in $(p_j,\exp(\omega_j.q))_{j=1\dots n}$. So now on $\Gamma$, the equation $\{I,H\}=J$ becomes $\dot{I}=J$, which gives us the Lemma.
\end{proof}

Remark that the problematic orbits $\Gamma$ are the poles of $I$. These are either the poles of $J$ (such a case could be seen immediately when computing $J(\Gamma(t))$), or the poles of $I$ which are not a poles of $J$. This is possible, but only in the case the pole is an invariant set of codimension $1$, i.e. $I$ has a Darboux function in its denominator.

\begin{lem}\label{lem2}
The only integrable potentials of the form
$$V=e^{iq_1} \sum\limits_{j \hbox{ finite}} a_j e^{ijq_2}$$
are up to symmetry $q_2\rightarrow -q_2$
$$V=a e^{i(q_1+kq_2)}+b e^{i(q_1-1/kq_2)} \quad V=a e^{i(q_1+\sqrt{3} q_2)}+b e^{i(q_1-\sqrt{3} q_2)} $$
$$V=a e^{i(q_1+3 q_2)}+b e^{i(q_1-2 q_2)} \quad V=a e^{i(q_1+3\sqrt{3} q_2)}+b e^{i(q_1-\frac{5}{3}\sqrt{3} q_2)} $$
\end{lem}

Let us remark that the angle between the two frequencies of the integrable cases are respectively $\pi/2,2\pi/3,3\pi/4,5\pi/6$.

\begin{proof}
Let us first remark that if the sum contains only one term, the Lemma is trivially satisfied, so we exclude this case from now on. Let us note $k=\max_{j\in\mathbb{Z}} \{j,\;a_j\neq 0\}$, and $s=\max_{j\in\mathbb{Z},\;j<k} \{j,\;a_j\neq 0\}$. Let us remark moreover that we can assume $k\neq 0$ by doing the symmetry $q_2\rightarrow -q_2$. We now use the transformation $\varphi_\alpha:q\rightarrow q-i\ln\alpha (0,1,0\dots,0)$. The weight homogeneous part of maximal degree of $V(\varphi_\alpha(q))$ is now $\tilde{V}=\exp(iq_1+ikq_2)$.

The potential $\tilde{V}$ admits the first integrals $H_0,kp_1-p_2,p_3,\dots,p_n$ where
$$H_0=\frac{1}{2}\sum\limits_{j=1}^n p_j^2 + e^{iq_1+ikq_2}.$$
All the rational first integrals of $\tilde{V}$ in $K$ are rational combinations of these ones. We now consider the next term in the series expansion in $\alpha$ of $V(\varphi_\alpha(q))$, which gives
$$H(\alpha^{(k+1)/2}p,\varphi_\alpha(q))=\alpha^k\left(\frac{1}{2}\sum\limits_{j=1}^n p_j^2 + e^{iq_1+ikq_2}\right)+ a\alpha^se^{iq_1+isq_2}+o(\alpha^s)$$
As the Hamiltonian $H$ is integrable, there exist a function $I\in K$ such that
\begin{equation}\label{eqintseries}
\{I(\alpha^{(k+1)/2}p,\varphi_\alpha(q)),H(\alpha^{(k+1)/2}p,\varphi_\alpha(q))\}=0\qquad \forall \alpha\in\mathbb{C}
\end{equation}
Looking at the first term in the series expansion at $\alpha=\infty$ of this expression, we have $\{I_0,H_0\}=0$ where $I_0$ is the weight homogeneous part of maximal degree of $I$. So $I_0$ is a first integral of $\tilde{V}$. As the Hamiltonian $H$ is invariant by translation in $q_j,\;j=3\dots n$, we can assume up to algebraic combination between $I$ and $H,(p_j)_{j=3\dots n}$ that $I$ does not depend on $(q_j,p_j)_{j=3\dots n}$. The $I_0$ is a first integral of $H_0$, and using Ziglin Lemma, we can assume that $I_0,H_0$ are functionally independent. Recalling that $I_0$ has to be a function of $kp_1-p_2,H_0$, up to algebraic transformations, we can assume that $I_0=kp_1-p_2$.

Let us now look at the second term of series of \eqref{eqintseries}. We have
$$\{I_0,ae^{iq_1+isq_2}\}+\{H_0,I_1\}=0$$
where $I_1$ is the next term in the series expansion of $I(\alpha^{(k+1)/2}p,\varphi_\alpha(q))$ at $\alpha=\infty$. This equation is of the type of Lemma \ref{lemint}. The only unknown is $I_1\in K$.

To apply Lemma \ref{lemint}, we first need to integrate the potential $\tilde{V}$, as we need a generic solution. The variables $q_j,p_j,\;j=3\dots n$ can be forgotten as the do not appear either in $H$ nor $I$. Let us note
$$\mathcal{M}_{\lambda,\mu}= \left\lbrace\frac{1}{2}(p_1^2+p_2^2)+e^{iq_1+ikq_2}=\frac{k^2\mu^2+\lambda^2}{2(k^2+1)},\;kp_1-p_2=\mu k\right\rbrace$$
This manifold is an invariant manifold. This manifold contains a generic orbit $\Gamma$ for a generic $\lambda,\mu$ as $k\in\mathbb{R}^*$. We now introduce the variables
$$e_1=e^{iq_1},\;\;e_2=e^{i\frac{k^2+1}{k}q_2}$$
We now parametrize the manifold $\mathcal{M}_{\lambda,\mu}$ by $p_1,e_2$
$$e_1=-\frac{k^4\mu^2-2k^4\mu p_1+k^4 p_1^2-2k^2\mu p_1+2k^2 p_1^2-\lambda^2+p_1^2}{2e_2^{\frac{k^2}{1+k^2}}(k^2+1)},\quad p_2=kp_1-k\mu.$$
On the manifold, there are two commuting vector fields, which lead to two rational closed $1$-forms whose integral $f_1,f_2$ are such that
$$L_{X_{H_0}} f_1=0,\; L_{X_{I_0}} f_1=1,\; L_{X_{H_0}} f_2=1,\; L_{X_{I_0}} f_2=0$$
Computing this forms and integrating them leads to
$$f_1=-\frac{ik}{k^2+1}\ln e_2+\frac{(\lambda-\mu)ik}{\lambda(k^2+1)} \ln((p_1-\mu)k^2-\lambda+p_1) +\frac{(\lambda+\mu)ik}{\lambda(k^2+1)}\ln((p_1-\mu)k^2+\lambda+p_1)$$
$$f_2=\frac{i}{\lambda} \ln((p_1-\mu)k^2-\lambda+p_1)-\frac{i}{\lambda}\ln((p_1-\mu)k^2+\lambda+p_1)$$
From this, we obtain the expression in time of a solution in $\mathcal{M}_{\lambda,\mu}$
$$e_2= \frac{4e^{it\lambda+i\mu t}\lambda^2}{(e^{it\lambda}-1)^2},\quad p_1 = \frac{e^{it\lambda}k^2\mu-k^2\mu+e^{it\lambda}\lambda+\lambda}{e^{it\lambda}k^2-k^2+e^{it\lambda}-1}$$
Remark that if $\lambda/\mu \notin \mathbb{Q}$ this solution is generic as almost all other solutions (so except the solutions in the singular sets of $f_1,f_2$) can be obtained by multiplying $e^{i\lambda t}, e^{i\mu t}$ by two non zero constants which is the Galois group action on this solution.

Let us now evaluate
$$\{H_0,I_1\}=\{ae^{iq_1+isq_2},kp_1-p_2\}$$
on such a solution. The righthandside evaluated on $\Gamma$ gives
\begin{equation}\label{eq1}
\tilde{a} \left( \frac{e^{it(\lambda+\mu)}}{(e^{it\lambda}-1)^2}\right)^{\frac{ks+1}{k^2+1}} e^{-i\mu t}
\end{equation}
with $\tilde{a}$ an always non zero constant multiple of $a$. We now try to integrate this expression in $K_{\mid_\Gamma}$, which is here
$$\mathbb{C}\left(\left( \frac{e^{it(\lambda+\mu)}}{(e^{it\lambda}-1)^2}\right)^{\frac{ks+1}{k^2+1}},e^{i\lambda t},e^{i\mu t}\right).$$
In the general case, this can be written with a hypergeometric function in $e^{i\lambda t},e^{i\mu t}$. This hypergeometric function becomes rational if and only if
\begin{equation}\label{eqtre}
\frac{ks+1}{k^2+1} \in -\frac{1}{2}\mathbb{N}.
\end{equation}
And it appears that in this case, the integral is indeed in $K_{\mid_\Gamma}$. We thus obtain
$$s=-\frac{k^2n+n+2}{2k},\;\;n\in\mathbb{N}$$
Now we remark that the angle between $(1,k)$ and $(1,s)$ is always larger than $\pi/2$. Applying our condition after making the symmetry $q_2\rightarrow -q_2$ (remark that the case $k=0$ or $s=0$ is now forbidden by our condition), we obtain that if $V$ has $3$ terms or more, then the angular separation between the extreme frequency vectors should be larger than $\pi$. This is impossible as they all are on the line $(1,y),\;y\in\mathbb{R}$. So $V$ has exactly $2$ terms. Applying our integrability condition simultaneously on $V$ and $V(q_1,-q_2)$ gives
$$\frac{ks+1}{k^2+1}= -\frac{n}{2},\quad \frac{ks+1}{s^2+1}=-\frac{m}{2},\quad n,m\in\mathbb{N}$$
We find that if this system admits a solution with $k\in\mathbb{R}$ then $nm<4$. We then find all the solutions
$$(k,-1/k),\;(\sqrt{3},-\sqrt{3}),\;(3,-2),\;(3\sqrt{3},-\frac{5}{3}\sqrt{3})$$
\end{proof}

The integrability condition \eqref{eqtre} is exactly the one given in \cite{17}, but now only assuming rational first integrals. Remark that a key argument in the proof is that $\exp(iq_1+ikq_2)$ is not superintegrable. Typically we have to consider the angular momenta, and so many other first integrals exist outside those in $p$ only. However, here we have to recall that the field $K$ consist only of $\mathcal{L}$-multiperiodic functions in $q$ and rational in $p$. The angular momenta are $p_iq_j-p_jq_i$ and none of their combinations (which has to be rational as the result has to be rational in $p$) produces a $\mathcal{L}$-multiperiodic function in $q$.

Remark also that the genericness of $\Gamma$ is important. In the function \eqref{eq1}, if the value $\mu=0$ is chosen, then for $s=-1/k$, the function \eqref{eq1} becomes a constant, and so its integral $\tilde{a}t$ is not in $K_{\mid_\Gamma}$. This seems to be related to the fact the perturbed system admits a first integral, but of higher degree in $p$.

\subsection{Integration in exponential fields}

In the next proofs, we will need to compute integrals of exponential functions, more precisely, of functions of the form
$$f(t)=\prod\limits_{j=1}^m e^{i\lambda_j\beta_j t} P(e^{i\lambda_1t},\dots,e^{i\lambda_mt}) \prod\limits_{j=1}^n Q_j(e^{i\lambda_1t},\dots,e^{i\lambda_mt})^{\alpha_j}$$
with $P\in\mathbb{C}[X_1,\dots,X_m]$, $Q_j\in\mathbb{C}[X_1,\dots,X_m]$ distinct unitary irreducible polynomials not monomial, $\alpha_j\in\mathbb{C}\setminus \mathbb{N},\beta_j\in\mathbb{C}$ and $\lambda_1,\dots,\lambda_m$ non resonant (i.e. there are no non zero integer relations between the $\lambda_j$).
Thanks to the non resonance condition, the function $f(t)$ can be uniquely written as a $m$ variable function of the form
$$\tilde{f}(X_1,\dots,X_m)=\prod\limits_{j=1}^m X_j^{\beta_j} P(X_1,\dots,X_m) \prod\limits_{j=1}^n Q_j(X_1,\dots,X_m)^{\alpha_j}$$
evaluated on $(e^{i\lambda_1t},\dots,e^{i\lambda_mt})$. A weighted degree defined on $\mathbb{C}[X_1,\dots,X_m]$ can be extended to a function of the form of $\tilde{f}$ by
$$\deg \tilde{f}= \sum\limits_{j=1}^m \beta_j \deg X_j+\deg P+ \sum\limits_{j=1}^n \alpha_j\deg Q_j.$$ 

\begin{lem}\label{lemint2}
If the function
$$f(t)= \prod\limits_{j=1}^m X_j^{\beta_j} P(e^{i\lambda_1t},\dots,e^{i\lambda_mt}) \prod\limits_{j=1}^n Q_j(e^{i\lambda_1t},\dots,e^{i\lambda_mt})^{\alpha_j}$$
admits an integral in $\mathbb{C}\left(f(t),e^{i\lambda_1t},\dots,e^{i\lambda_mt}\right)$, then
\begin{itemize}
\item We can write
\begin{equation}\label{eqint}
\int f(t) dt= \prod\limits_{j=1}^m e^{i\lambda_j\beta_j t} G(e^{i\lambda_1t},\dots,e^{i\lambda_mt}) \prod\limits_{j=1}^n Q_j(e^{i\lambda_1t},\dots,e^{i\lambda_mt})^{\alpha_j+1}
\end{equation}
with $G\in\mathbb{C}[X_1,\dots,X_m,1/X_1,\dots,1/X_m]$.
\item $\alpha_j\neq -1 \forall j=1\dots n$.
\item For any integer weighted degree (respectively valuation) such that $\deg \tilde{f} \notin \mathbb{Z}$ (respectively $\hbox{val} \tilde{f} \notin \mathbb{Z}$), we have respectively
$$\deg G=\deg P-\sum\limits_{j=1}^n \deg Q_j \qquad \hbox{val} G=\hbox{val} P-\sum\limits_{j=1}^n \hbox{val} Q_j$$
\end{itemize}
\end{lem}

\begin{proof}
Let us first remark that the derivation in $t$ for $f$ corresponds to the derivation $D=\sum_{j=1}^m \lambda X_j \partial_{X_j}$ for $\tilde{f}$. This derivation $D$ only admits as Darboux polynomials the monomials. If $f$ admits an integral in $\mathbb{C}(f(t),e^{i\lambda_1t},\dots,e^{i\lambda_mt})$, then we have a $m$-variable function $\tilde{g}$ such that $D\tilde{g}=\tilde{f}$. The poles and ramifications loci of $\tilde{g}$ outside the coordinates planes $X_j=0$ are thus the poles/ramification loci of $\tilde{f}$ with one order higher. If all the $\alpha_j,\beta_j$ are integers, then $f(t)\in \mathbb{C}(e^{i\lambda_1t},\dots,e^{i\lambda_mt})$, and so equation \eqref{eqint} is automatically satisfied. If at least one $\alpha_j,\beta_j$ is not an integer, then the Galois group acts multiplicatively on $\tilde{f}$, and then also on $\tilde{g}$. This implies we can write
\begin{equation}\label{eqint2}
\tilde{g}= \prod\limits_{j=1}^m X_j^{\beta_j} G(X_1,\dots,X_m) \prod\limits_{j=1}^n Q_j(X_1,\dots,X_m)^{\alpha_j+1}
\end{equation}
with $G\in\mathbb{C}[X_1,\dots,X_m,1/X_1,\dots,1/X_m]$. This gives the first property of the Lemma. As the $Q_j$ are poles/ramifications points of $\tilde{g}$, $\alpha_j+1\notin \mathbb{N}$, and thus $\alpha_j\neq -1$, giving the second property.

Let us now prove the degree/valuation condition on $G$. Let us first remark that if $\deg \tilde{g} \neq 0$, then $\deg \tilde{g}= \deg D\tilde{g}$ (if the dominant term is a constant, the derivation vanishes it and the degree decreases). Due to the first proposition of the Lemma, we already know that the degree of $\tilde{g}$ and $\tilde{f}$ differ by an integer and thus if $\deg \tilde{f} \notin \mathbb{Z}$, we have $\deg \tilde{g} \neq 0$ and thus $\deg \tilde{g}= \deg D\tilde{g}=\deg \tilde{f}$. Removing the degrees of the $Q_j$ and $X_j$ factors, we obtain
$$\deg G = \deg P-\sum\limits_{j=1}^n \deg Q_j.$$
The same reasoning holds for a weighted valuation.
\end{proof}

Remark that the condition $\deg \tilde{f} \notin \mathbb{Z}$ cannot be satisfied if $\tilde{f}$ is rational. This implies that the last criterion can only be effective if at least one of the $\alpha_j,\beta_j$ is not an integer. The Lemma can be turn into an effective integration algorithm if this condition is satisfied for sufficiently many valuations and degrees, and so allowing to bound the monomial support of $G$. If the non integer degree condition does not hold, additional zero degree terms can appear in the integral, and so a larger linear system need to be solved, as in the following
$$\int \frac{2e^{2t}}{(e^{2t}+1)^2} dt=-\frac{1}{e^{2t}+1}$$
We have $\tilde{f}=2X^2/(X^2+1)^2$ with valuation $2$ and degree $-2$. However, the result is of valuation $0$ and degree $-2$. And just by adding a zero degree term $1$, we obtain $X^2/(X^2+1)$ which has valuation $2$ and degree $0$. So only after adding a zero degree term (which can depend on the degree/valuation chosen), the third proposition of the Lemma can be satisfied.

\subsection{Support of $\hat{V}$}

\begin{lem}\label{lemsummit}
Let $V$ be an integrable trigonometric polynomial potential, and $v$ a summit of $\mathcal{C}(V)$. Then
$$\mathcal{S}(V) \cap \{k\in\mathcal{L},\; k.v>0\} \subset \mathbb{R}_+.v$$
\end{lem}

\begin{proof}
Let us first remark that using Lemma \ref{lem2}, all the other summit with an edge going to $v$ should have an angular separation of at least $\pi/2$. This implies in particular that the half space $\{k\in\mathbb{R}^n,\; k.v>0\}$ contains no other summit of $\mathcal{C}(V)$. Let us now make a rotation/dilatation/multiplication by a non zero constant to put $v$ at $(1,0,\dots,0)$ and the corresponding coefficient to $1$. Let us note
$$\gamma= \max \{s, \exists (l_2,\dots,l_n) \hbox{ not all zero such that } (s,l_2,\dots,l_n)\in \mathcal{S}(V) \}.$$
We have $\gamma<1$. If $\gamma\leq 0$, then the Lemma is proved. So we can assume $\gamma \in ]0,1[$. We now apply $\varphi_\alpha$. The dominant term of the Hamiltonian is
$$H_0=\frac{1}{2}\sum\limits_{j=1}^n p_j^2 +e^{iq_1}$$
This Hamiltonian admits the first integrals $p_1^2/2+e^{iq_1},p_2,\dots,p_n$ and all the other ones in $K$ are rational combinations of these. Let us now note $I_2,\dots,I_n$ commuting independent first integrals of $H$, and $I_{2,0},\dots,I_{n,0}$ their dominant term after the transformation with $\varphi_\alpha$. Using Ziglin Lemma and then algebraic transformations, we can assume that $I_{j,0}=p_j,\;j=2,\dots,n$. We now make an expansion of $H(\alpha^{1/2}p,\varphi_\alpha(q))$ up to order $\gamma$ at $\alpha=\infty$, giving
$$\frac{\alpha}{2}\sum\limits_{j=1}^n p_j^2 +\alpha e^{iq_1}+\alpha^{r_1} s_1 e^{ir_1q_1}+\dots +\alpha^{r_p} s_p e^{ir_pq_1}+e^{i\gamma q_1}\sum\limits_{j=1}^m a_je^{\delta_j.(q_2,\dots,q_n)} +o(\alpha^\gamma) $$
with $s_j\neq 0$, $r_j$ strictly decreasing sequence in $]\gamma,1[$, $a_j\neq 0$ and $\delta_j\in \mathbb{R}^{n-1}$. Let us note $H_1$ the last sum part (the one of order $\gamma$). If we stop at order $r_p$, the Hamiltonian has still $p_2,\dots,p_n$ as first integrals. We can thus also assume that the first integrals $I_2,\dots,I_n$ are equal to $p_2,\dots,p_n$ up to order $r_p$. Noting $I_{2,1},\dots, I_{n,1}$ the term of order $\gamma$ in $I_2,\dots,I_n$, we have
$$\{H_0,I_{l,1}\}+\left\lbrace H_1 ,p_l\right\rbrace=0$$
We now want to apply Lemma \ref{lemint}. A generic solution of $H_0$ can be written
$$e^{Iq_1}=\frac{2\lambda^2}{\cos(t\lambda)^2},\; p_1^2/2+e^{iq_1}=2\lambda^2,\; (p_j=c_j,q_j=c_jt)_{j=2\dots n}$$
where $c\in\mathbb{R}^{n-1}$ an arbitrary vector and $\lambda\in\mathbb{C}^*$. Now on this solution, the functions we have to integrate with respect to time are

$$\left(\frac{8\lambda^2e^{2i\lambda t}}{(e^{2i\lambda t}+1)^2}\right)^\gamma \sum\limits_{j=1}^m a_j\delta_{j,l}e^{\delta_j.c t}\qquad l=2\dots n$$
We now use Lemma \ref{lemint}. The field $K_{\mid_\Gamma}$ is
$$\mathbb{C}(\cos(t\lambda)^{-2r_1},\dots,\cos(t\lambda)^{-2r_p},\cos(t\lambda)^{-2\gamma},(e^{\delta_j.c t})_{j=1\dots m}).$$
To integrate the function, we use Lemma \ref{lemint2}. As $c$ can be chosen arbitrary, we can assume $\lambda,(\delta_j.c)_{j=1\dots m}$ non resonant. A necessary condition is $2\gamma\neq 1$ except if the whole sum is zero. The integral is searched under the form
$$e^{2i\lambda\gamma t}(e^{2i\lambda t}+1)^{-2\gamma+1} G(e^{i\lambda t},(e^{\delta_j.c t})_{j=1\dots m}) $$
We now compute the degree and valuation in $e^{i\lambda t}$ of the integrand, giving respectively $-2\gamma,2\gamma$. As these are not integers, the third proposition of Lemma \ref{lemint2} applies, and gives that the degree and valuation of $G$ should be $-2,0$, which is impossible. Thus the whole sum is $0$.\\

So we have
$$\sum\limits_{j=1}^m a_j\delta_{j,l}e^{\delta_j.c t}=0,\quad l=2\dots n$$
and as the $\delta_j.c$ are non resonant, this implies that
$$a_j\delta_{j,l}=0,\;\;\forall j,l$$
So $a_j=0$ except if $\delta_{j,l}=0,\; l=2\dots n$. So only the zero frequency vector is allowed. But then $H_1$ reduces to a constant times $e^{i\gamma q_1}$. This is impossible due to the definition of $\gamma$.
\end{proof}

\subsection{Complete rational integrability}

In the next section, we will have to explicitly integrate the following $3$ Hamiltonian systems with high degree first integrals

\begin{small}\begin{align*}
&H_1=\frac{1}{2}(p_1^2+p_2^2)+e^{i(q_1+\sqrt{3}q_2)}+e^{i(q_1-\sqrt{3}q_2)}\\
&I_1=p_2^3-3p_2p_1^2+3\sqrt{3}\left(e^{i(q_1-\sqrt{3}q_2)}-e^{i(q_1+\sqrt{3}q_2)}\right)p_1+3\left(e^{i(q_1+\sqrt{3}q_2)}+e^{i(q_1-\sqrt{3}q_2)}\right)p_2\\
&H_2=\frac{1}{2}(p_1^2+p_2^2)+e^{iq_1}+e^{i(q_2-q_1)}\\
&I_2=p_1^2p_2^2+2p_2^2e^{iq_1}-2p_1p_2e^{i(q_2-q_1)}+2e^{iq_2}+e^{2i(q_2-q_1)}\\
&H_3=\frac{1}{2}(p_1^2+p_2^2)+e^{2i\sqrt{3}q_1}+e^{-i\sqrt{3}q_1-iq_2}\\
&I_3=-p_1^6+6p_2^2p_1^4-9p_2^4p_1^2-6(e^{2i\sqrt{3}q_1}+e^{-i\sqrt{3}q_1-iq_2})p_1^4+6\sqrt{3}e^{-i\sqrt{3}q_1-iq_2}p_2p_1^3+\\
& \left(24e^{2i\sqrt{3}q_1}+18e^{-i\sqrt{3}q_1-iq_2}\right)p_2^2p_1^2-18\sqrt{3}e^{-i\sqrt{3}q_1-iq_2}p_2^3p_1-18e^{2i\sqrt{3}q_1}p_2^4+\\
& 6\sqrt{3}\left(3e^{-2i(\sqrt{3}q_1+q_2)}+2e^{i\sqrt{3}q_1-iq_2}\right)p_2p_1-3\left(3e^{-2i(\sqrt{3}q_1+q_2)}+8e^{i\sqrt{3}q_1-iq_2}+4e^{4i\sqrt{3}q_1}\right)p_1^2\\
& +\left(-27e^{-2i(\sqrt{3}q_1+q_2}+36e^{i\sqrt{3}q_1-iq_2}+24e^{4i\sqrt{3}q_1}\right)p_2^2- 8e^{6i\sqrt{3}q_1}-24e^{3i\sqrt{3}q_1-iq_2}
\end{align*}\end{small}
The technique of separation of variable cannot be applied, however their Liouville tori have very good algebraic property, suggesting the following definition.

\begin{defi}
An integrable rational Hamiltonian $H(p,q),\; (p,q)\in\mathbb{C}^{2n}$ with commuting first integrals $I_1=H,\dots,I_n \in\mathbb{C}(p,q)$ is said to be completely rationally integrable if on a generic Liouville tori $\mathcal{M}=\{I_1=j_1,\dots,I_n=j_n\}$, there exists a birational transformation which maps $\mathcal{M}$ to $\mathbb{P}^n$ and the commuting Hamiltonian fields $J\nabla I_i$ to linear commuting fields on $\mathbb{P}^n$.
\end{defi}

This notion of integrability is different of algebraic complete integrability, where the Liouville tori are Abelian varieties, but not necessary rational \cite{8}. Remark that a generic solution of a completely rationally integrable system can be written just using at most $n$ exponential functions of the time. The coefficients in these exponentials are the periods of the corresponding Liouville tori, and due to the fact they can be found be a residue computation, they are algebraic functions of the first integral levels. Resonance of a Liouville tori corresponds to the case when the number of necessary exponential functions is not maximal, for example if there is a resonance between the frequencies of the exponential functions, and super-integrability when this resonance is generic.

Remark also that the birational transformation is not unique. We can first reduce the linear vector fields on $\mathbb{P}^n$ by diagonalization. Now we can arbitrary scale the axes (which corresponds to the action of the vector field on the Liouville tori), but also consider the maps
$$\phi(x_1,\dots,x_n)=\left(\prod\limits_{i=1}^n x_i^{a_{ij}}\right)_{j=1\dots n}\quad (a_{ij})_{i,j=1\dots n} \in GL_n(\mathbb{Z})$$
These maps are birational, and conserve the axes and the linearity of the vector field. However, they change the frequencies, and correspond to automorphisms of the Liouville tori. This implies that the frequencies of the exponentials are not themselves intrinsic to the Hamiltonian system but only their generated lattice.

In terms of action angle coordinates, the coordinates system given by the birational transformation are the exponentials of $2i\pi\theta_i$ where $\theta_i$ are the angle coordinates. Thus the Arnold Liouville Theorem gives that if the real part of $\mathcal{M}$ is compact, its image is the product of unit circles of $\mathbb{P}^n$, which is indeed a torus.\\

To check if our $3$ Hamiltonians are completely rationally integrable, we first need to represent them as rational functions. So we note $e_1,e_2$ the two exponentials in the Hamiltonians. The Hamiltonians $H_i$ are rational in $p,e$, and so are the first integrals and commuting Hamiltonian field $J\nabla I_i$. We will now perform the following integration procedure to determine if the $H_i$ are indeed completely rationally integrable.
\begin{itemize}
\item The manifold $H=h,I=j$ is a $2$-dimensional algebraic manifold depending on parameters $h,j$ (the level of the first integrals). We need to rationally parametrize this manifold, as it should be birationally equivalent to $\mathbb{P}^2$ \cite{7}. Now rational functions on $H=h,I=j$ can be represented as elements of $\tilde{K}=\overline{\mathbb{C}(h,j)}(u,v)$.
\item We compute the Hamiltonian vector fields associated to $H_i,I_i$ expressing them in $u,v$, giving two vector fields in dimension $2$ with coefficients in $\tilde{K}$.
\item We solve the partial differential equation systems
$$L_{X_{H_i}} f_1=0,\; L_{X_{I_i}} f_1=1,\; L_{X_{H_i}} f_2=1,\; L_{X_{I_i}} f_2=0$$
We have to integrate a closed $1$-form with coefficients in $\tilde{K}$. This can be done by absolute factorization and Hermite decomposition \cite{9,10}.
\item The solutions $p(t),e(t)$ have to be rational functions in at most $2$ exponential functions in time (as there should exist a birational transformation mapping the Hamiltonian fields to linear ones), and so are the $u(t),v(t)$. This implies that the system $f_1=c_1,f_2=t+c_2$ can be inverted using at most $2$ exponential functions in time. So the $f_1,f_2$ have to be sums of only logs, and the $\mathbb{Q}$ vector space of the residues has to be of dimension $\leq 2$.
\end{itemize}

\begin{prop} \label{propratint}
The $3$ Hamiltonians $H_1,H_2,H_3$ are completely rationally integrable.
\end{prop}

\begin{proof}

\textbf{Case $H_1$}. We parametrize the invariant manifolds by
$$\mathcal{M}_{\lambda,\mu}=\left\lbrace H_1=\frac{1}{6}(\lambda^2-\lambda\mu+\mu^2),\; I_1=\frac{\sqrt{3}}{18}(\lambda-2\mu)(2\lambda-\mu)(\lambda+\mu)\right\rbrace$$
For $\lambda,\mu\in\mathbb{C}$ we can recover the entire phase space with such manifolds. This manifold can simply be rationally parametrized by $p_1,p_2$ giving
\begin{scriptsize}\begin{equation}\begin{split}\label{eqint1H1}
e_1 &=-\frac{(\sqrt{3}p_2+2\lambda-\mu+3p_1)(\sqrt{3}p_2-\lambda+2\mu+3p_1)(\sqrt{3}p_2-\lambda-\mu+3p_1)}{108p_1}\\
e_2 &= \frac{(\sqrt{3}p_2-\lambda+2\mu-3p_1)(\sqrt{3}p_2-\lambda-\mu-3p_1)(\sqrt{3}p_2+2\lambda-\mu-3p_1)}{108p_1}
\end{split}\end{equation}\end{scriptsize}
where $e_1=\exp(i(q_1+\sqrt{3}q_2)),e_2=\exp(i(q_1-\sqrt{3}q_2))$. Now we compute the vector fields on $\mathcal{M}_{\lambda,\mu}$ and solve the PDEs for $f_1,f_2$. They can be expressed as a sum of logs, and we observe that there are only $2$ different residues for $f_1,f_2$. So after linear transformation and then exponentiation, the system $f_1=c_1,f_2=t+c_2$ becomes an algebraic system in $p_1,p_2$ and $e^{i\lambda t},e^{i\mu t}$. For $c_1=c_2=0$, this gives
\begin{scriptsize}\begin{equation}\begin{split}\label{eqint2H1}
p_2 &=\frac{2x(y-1)^2\lambda^3+\mu(y-1)(x^2-3xy+3x-y)\lambda^2-\mu^2(x-1)(3xy-y^2+x-3y)\lambda+2\mu^3y(x-1)^2}{2\sqrt{3}((y-1)\lambda-\mu(x-1))(x(y-1)\lambda-\mu y(x-1))}\\
p_1 &=\frac{(y-1)(x-1)(x-y)(\lambda-\mu)\lambda\mu}{2(\mu(x-1)-(y-1)\lambda)(x(y-1)\lambda-\mu y(x-1))}
\end{split}\end{equation}\end{scriptsize}
where $x=\exp(i\lambda t),y=\exp(i\mu t)$. The Galois action on the exponentials generates the other solutions in $\mathcal{M}_{\lambda,\mu}$. The expression of $e_1,e_2$ can then be obtained by substitution in \eqref{eqint1H1}. The birational transformation of the definition  $\phi_{\lambda,\mu}:\mathbb{P}^2 \mapsto \mathcal{M}_{\lambda,\mu}$ is simply the expression of $p_1,p_2,e_1,e_2$ in function of $x,y$, which is birational by construction.\\

\textbf{Case $H_2$}. We parametrize the invariant manifolds by
$$\mathcal{M}_{\lambda,\mu}=\left\lbrace H_2=\frac{1}{2}(\lambda^2+\mu^2), I_2=\lambda^2\mu^2\right\rbrace$$
This manifold can be rationally parametrized by two variables $(u,v)$
\begin{scriptsize}\begin{equation}\begin{split}\label{eqint1H2}
e_2 &=\frac{(u^2-\lambda^2+\mu^2+2\mu u-v)(\lambda^2-\mu^2+2\mu u-u^2+v)(\lambda^2-2\lambda u-\mu^2+u^2-v)(\lambda^2+2\lambda u-\mu^2+u^2-v)}{16u^2v^2}\\
p_1 &=\frac{\lambda^4-2\lambda^2\mu^2-2\lambda^2u^2+\mu^4-2\mu^2u^2+u^4+4u^2v-v^2}{4uv}\\
p_2 &=-\frac{\lambda^4-2\lambda^2\mu^2-2\lambda^2u^2+\mu^4-2\mu^2u^2+u^4-v^2}{4uv}\\
e_1 &=-\frac{(u+\lambda+\mu)(\lambda+\mu-u)(\lambda-\mu-u)(\lambda-\mu+u)}{2v}
\end{split}\end{equation}\end{scriptsize}
where $e_1=\exp(iq_1),\;e_2=\exp(i(q_2-q_1))$. Now we compute the vector fields on $\mathcal{M}_{\lambda,\mu}$ and solve the PDEs for $f_1,f_2$. We have again only $2$ different residues, so after linear transformation and then exponentiation, the system $f_1=c_1,f_2=t+c_2$ becomes an algebraic system in $u,v$ and $e^{i\lambda t},e^{i\mu t}$. For $c_1=c_2=0$, this gives
\begin{scriptsize}\begin{equation}\begin{split}\label{eqint2H2}
v &=\frac{4(\mu-\lambda)(\lambda+\mu)(\lambda y-\mu x-\lambda+\mu)(\lambda xy-\lambda x-\mu x+\mu)(\lambda xy-\mu xy-\lambda x+\mu y)(-\mu xy+\lambda y+\mu y-\lambda)}
{(\lambda^2xy^2-\mu^2x^2y-2\lambda^2xy+2\mu^2xy+\lambda^2x-\mu^2y)(\lambda xy-\mu xy-\lambda x+\lambda y-\mu x+\mu y-\lambda+\mu)^2}\\
u &=\frac{(\lambda-\mu)(\lambda+\mu)(y-1)(x-1)}{\lambda xy-\mu xy-\lambda x+\lambda y-\mu x+\mu y-\lambda+\mu}
\end{split}\end{equation}\end{scriptsize}
where $x=\exp(i\lambda t),y=\exp(i\mu t)$. The Galois action on the exponentials generates the other solutions in $\mathcal{M}_{\lambda,\mu}$. The expression of $p_1,p_2,e_1,e_2$ can then be obtained by substitution in \eqref{eqint1H2}. The birational transformation of the definition is simply the expression of $p_1,p_2,e_1,e_2$ in function of $x,y$.\\

\textbf{Case $H_3$}. We parametrize the invariant manifolds by
$$\mathcal{M}_{\lambda,\mu}=\left\lbrace H_3=\frac{1}{6}(\lambda^2+\lambda\mu+\mu^2),I_3=-\frac{1}{108}(\lambda-\mu)^2(2\mu+\lambda)^2(\mu+2\lambda)^2\right\rbrace$$
This manifold can be rationally parametrized by two variables $(u,v)$
\begin{scriptsize}\begin{equation}\begin{split}\label{eqint1H3}
p_1 &=\frac{u^6-(18\lambda^3v^3+27\lambda^2\mu v^3-27\lambda\mu^2v^3-18\mu^3v^3+3\lambda^2v^2+3\lambda\mu v^2+3\mu^2v^2+3)u^4-Q(v)u^2-P(v)^3}
{4\sqrt{3}uv(u^4+(8\lambda^3v^3+12\lambda^2\mu v^3-12\lambda\mu^2v^3-8\mu^3v^3-6\lambda^2v^2-6\lambda\mu v^2-6\mu^2v^2-2)u^2+P(v))}\\
e_1 &=\frac{u^2v(\mu-\lambda)(2\mu+\lambda)(\mu+2\lambda)((2\lambda v+\mu v+1)^2-u^2)((\lambda v+2\mu v-1)^2-u^2)((\lambda v-\mu v-1)^2-u^2)}
{3(u^4+(8\lambda^3 v^3+12\lambda^2\mu v^3-12\lambda\mu^2v^3-8\mu^3v^3-6\lambda^2v^2-6\lambda\mu v^2-6\mu^2v^2-2)u^2+P(v)^2)^2}\\
e_2 &=-\frac{u^2}{72v^2}-\frac{4\lambda^3v^3+6\lambda^2\mu v^3-6\lambda\mu^2 v^3-4\mu^3v^3-3\lambda^2v^2-3\lambda\mu v^2-3\mu^2v^2-1}{36v^2}-
\frac{P(v)^2}{72v^2u^2}\\
p_2 &=\frac{u}{12v}-\frac{P(v)}{12uv}\\
\end{split}\end{equation}\end{scriptsize}
with
\begin{scriptsize}\begin{align*}
P(v) &=(\lambda v+2\mu v-1)(2\lambda v+\mu v+1)(\lambda v-\mu v-1)\\
Q(v) &=P(v)(14\lambda^3v^3+21\lambda^2\mu v^3-21\lambda\mu^2v^3-14\mu^3v^3-3\lambda^2v^2-3\lambda\mu v^2-3\mu^2v^2-3)
\end{align*}\end{scriptsize}
where $e_1=\exp(2i\sqrt{3}q_1)$, $e_2=\exp(-i\sqrt{3}q_1-iq_2)$. As before, we compute the functions $f_1,f_2$ and then solve the system $f_1=c_1,f_2=t+c_2$, which becomes algebraic in $u,v$ after linear transformation and exponentiation. We obtain an expression of $u,v\in\mathbb{C}(e^{i\lambda t},e^{i\mu t})$
\begin{scriptsize}\begin{equation}\begin{split}\label{eqint2H3}
u &=\frac{3(x-1)(y-1)(xy-1)(\lambda+\mu)\lambda\mu}{2\lambda^3x(y-1)^2-2y(x-1)^2\mu^3-\lambda(x-1)(xy^2+3xy-3y-1)\mu^2+\lambda^2(y-1)(x^2y+3xy-3x-1)\mu}\\
v &=\frac{2((1-x)\mu+\lambda x(y-1))(\lambda(y-1)-y(x-1)\mu)}{2\lambda^3x(y-1)^2-2y(x-1)^2\mu^3-\lambda(x-1)(xy^2+3xy-3y-1)\mu^2+\lambda^2(y-1)(x^2y+3xy-3x-1)\mu}
\end{split}\end{equation}\end{scriptsize}
where $x=\exp(i\lambda t),y=\exp(i\mu t)$. The Galois action on the exponentials generates the other solutions in $\mathcal{M}_{\lambda,\mu}$. The expression of $p_1,p_2,e_1,e_2$ can then be obtained by substitution in \eqref{eqint1H3}. The birational transformation of the definition is simply the expression of $p_1,p_2,e_1,e_2$ in function of $x,y$.\\
\end{proof}

Remark that the common level of first integrals $\mathcal{M}_{\lambda,\mu}$ is not chosen randomly. Indeed, when we compute the functions $f_1,f_2$, we have to integrate a closed $1$-form with rational coefficients (using the rational parametrization of the Liouville tori), and coefficient field extensions could appear. In practice, we first use arbitrary levels $H=h,I=j$, compute the $1$-forms and perform an absolute factorization of the denominator. This defines an algebraic field extension $\mathbb{L}$ over $\mathbb{Q}(h,j)$, which we manage to rationally parametrize, i.e. $\mathbb{L} \simeq \mathbb{Q}(\lambda,\mu)$ where $h,j$ are rational in $\lambda,\mu$. After computation, it appeared that these $\lambda,\mu$ can be chosen as frequencies for the Liouville tori, simplifying the computations.

\subsection{The angles $2\pi/3,3\pi/4,5\pi/6$}

\begin{lem}\label{leminterior}
Let $V$ be an integrable trigonometric polynomial potential with $v_1,v_2$ two summits of a same edge of $\mathcal{C}(V)$, with $\lVert v_1\lVert \geq \lVert v_2\lVert$. 
\begin{itemize}
\item If their angular separation is $2\pi/3,5\pi/6$, then
$$\mathcal{S}(V)\cap \mathbb{R}_+.v_i=\{v_i\},\;\;i=1,2$$
\item If their angular separation is $3\pi/4$, then
$$\mathcal{S}(V)\cap \mathbb{R}_+.v_1\subset \{v_1,v_1/2\},\quad \mathcal{S}(V)\cap \mathbb{R}_+.v_2=\{v_2\} $$
\end{itemize}
\end{lem}

The potential corresponding to the edge $v_1,v_2$ has to be one the potentials of Lemma \ref{lem2} up to rotation/dilatation. Remarking moreover that the constants $a,b$ cannot be $0$ (as $v_1,v_2\in \mathcal{S}(v)$), we can assume after rotation dilatation that the summits $v_1,v_2$ are
$$((1,\sqrt{3}),(1,-\sqrt{3})),\;((-1,1),(1,0)),\;((2\sqrt{3},0),(-\sqrt{3},-1))$$
for respectively the angles $2\pi/3,3\pi/4,5\pi/6$. Remark moreover that using Lemma \ref{lemsummit}, we know that in the half spaces
$$\{k\in\mathcal{L},\; k.v_1>0\},\; \{k\in\mathcal{L},\; k.v_2>0\}$$
all the frequencies belong to the plane $v_1,v_2$ and this will allow us to forget the variables $p_3,q_3,\dots,p_n,q_n$, reducing the analysis to two degrees of freedom. Now up to translation in $q_1,q_2$, as we know the coefficients associated to $v_1,v_2$ are not zero, we can assume these coefficients equal $1$. This gives the $3$ Hamiltonians $H_1,H_2,H_3$ of the previous section.

\begin{proof}
Let us begin by the first Hamiltonian $H_1$. Noting $e_1=\exp(iq_1+i\sqrt{3}q_2),\;e_2=\exp(iq_1-i\sqrt{3}q_2)$, the Hamiltonian $H_1$ and the first integral $I_1$ are rational in $p_1,p_2,e_1,e_2$. Now using Lemma \ref{lem2}, we know that for our potential $V$, we have
$$V(\varphi_\alpha(q_1,q_2))=\alpha(e_1+e_2)+ \alpha^\gamma(ae_1^\gamma+be_2^\gamma)+o(\alpha^\gamma)$$
where $\varphi_\alpha$ is the translation by $-i(1/\sqrt{3},0)\ln\alpha$, $\gamma\in ]0,1[$, $a,b\in\mathbb{C}$. We need to prove that $a=b=0$ is a necessary condition for integrability. This is again a perturbation of the Hamiltonian $H_1$, and so we want to use Lemma \ref{lemint}. So we use the expressions of the solutions \eqref{eqint1H1}, \eqref{eqint2H1}. If $\lambda/\mu\notin \mathbb{Q}$, the Galois action on this solution generates all the other solutions in $\mathcal{M}_{\lambda,\mu}$, ensuring that this solution is generic. We now compute $\{ae_1^\gamma+be_2^\gamma,I_1\}$ and evaluate on this generic solution. Up to multiplication by a non zero constant, the coefficients in $a,b$ are
\begin{scriptsize}\begin{equation*}\begin{split}
(\lambda y-\mu x-\lambda+\mu)^{-2\gamma-1}x^\gamma y^\gamma(\lambda xy-\mu xy-\lambda x+\mu y)^{\gamma-1}\times\\
(x(y-1)^2\lambda^4-\mu(y-1)(3xy-x+y)\lambda^3+\\
\mu^2(2x^2y+2xy^2+2x^2-6xy+2y^2-x-y)\lambda^2-\mu^3(x-1)(3xy+x-y)\lambda+\mu^4y(x-1)^2)\\
(\lambda y-\mu x-\lambda+\mu)^{-2\gamma-1}(\lambda xy-\mu xy-\lambda x+\mu y)^{\gamma-1}\times\\
(x(y-1)^2\lambda^4+\mu x(y-1)(x-y+3)\lambda^3-\\
\mu^2(x^2y+xy^2-2x^2+6xy-2y^2-2x-2y)\lambda^2-\mu^3y(x-1)(x-y-3)\lambda+\mu^4y(x-1)^2)
\end{split}\end{equation*}\end{scriptsize}
where $x=\exp{i\lambda t},\;y=\exp{i\mu t}$. As $\gamma$ is not an integer, these two terms are independent over $\mathbb{C}(x,y)$, and so we can try to compute an integral for each of them separately. For the first term, we compute the valuation in $x$, $y$, and the total degree, giving respectively $\gamma,\gamma,2\gamma$. These are not integers, so Lemma \ref{lemint2} applies, and we find that the valuation in $x$, $y$ and total degree of $G$ should be $0,0,0$, which implies that $G$ should be a constant. However trying $G$ constant does not lead to an integral.
For the second term, we compute the degree in $x$, $y$, and total valuation, giving respectively $-\gamma,-\gamma,\gamma$. Lemma \ref{lemint2} applies, and we find that the degree in $x$, $y$ and total valuation of $G$ should be $0,0,0$, which implies that $G$ should be a constant. However trying $G$ constant does not lead to an integral. Thus a necessary and sufficient condition for integrability is $a=b=0$. This gives case $1$ of the Lemma.\\

For the second Hamiltonian $H_2$, the solutions are \eqref{eqint1H2}, \eqref{eqint2H2}. We now compute $\{ae_1^\gamma+be_2^\gamma,I_2\}$ and evaluate on this generic solution. Up to multiplication by a non zero constant, the coefficients in $a,b$ are
\begin{scriptsize}\begin{equation*}\begin{split}
(\mu^2x^2y-\lambda^2xy^2+2\lambda^2xy-2\mu^2xy-\lambda^2x+\mu^2y)^{\gamma-1}(\lambda xy-\mu xy-\lambda x+\lambda y-\mu x+\mu y-\lambda+\mu)^{-2\gamma-1}\\
\times(\lambda xy-\mu xy+\lambda x-\lambda y+\mu x-\mu y-\lambda+\mu)(\mu x^2y-\lambda xy^2+\lambda x-\mu y)\\
(\mu^2x^2y-\lambda^2xy^2+2\lambda^2xy-2\mu^2xy-\lambda^2x+\mu^2y)^{-2\gamma-1}(\lambda xy-\mu xy-\lambda x+\lambda y-\mu x+\mu y-\lambda+\mu)^{2\gamma-1}\!\! x^\gamma y^\gamma\\
\times(x(y+1)(y-1)^2(x+1)\lambda^4-2\mu x(y-1)(y+1)^2(x-1)\lambda^3+\\
\mu^2(y+1)(x+1)(x^2y+xy^2-4xy+x+y)\lambda^2-2\mu^3y(x-1)(x+1)^2(y-1)\lambda+\mu^4y(x+1)(x-1)^2(y+1))
\end{split}\end{equation*}\end{scriptsize}
where $x=\exp{i\lambda t},\;y=\exp{i\mu t}$. As $\gamma$ is not an integer, these two terms are independent over $\mathbb{C}(x,y)$, and so we can try to compute an integral for each of them separately. For the first term, the degree and valuation in $x,y$ are all integers. So let us introduce the weighted degrees
$$\hbox{deg}_1(x^iy^j)=i+j,\quad \hbox{deg}_2(x^iy^j)=i-j$$
and corresponding valuations $\hbox{val}_1,\hbox{val}_2$. We then obtain for valuation/degree $1:[\gamma,-\gamma]$, and for $2:[\gamma,-\gamma]$. Now we find that the valuations/degrees of function $G$ should be zero, implying that $G$ is constant. After differentiation, we find a constant $G$ cannot give the integral.

For the second term, we compute the valuation/degree in $x: [\gamma,-\gamma]$ and in $y: [\gamma,-\gamma]$. These are not integers, so Lemma \ref{lemint2} applies. If $\gamma\neq 1/2$, we find that the valuation/degree in $x$ of $G$ should be $0,0$, and in $y$ $0,0$. This implies that $G$ should be a constant, but $G$ constant does not lead to an integral. If $\gamma=1/2$, one of the factors becomes polynomial, and so the formulas for the valuation/degree of $G$ changes. The term integrates, giving up to a constant factor
$$J=\frac{(\lambda xy-\mu xy+\lambda x-\lambda y+\mu x-\mu y-\lambda+\mu)\sqrt{xy}}{\mu^2x^2y-\lambda^2xy^2+2\lambda^2xy-2\mu^2xy-\lambda^2x+\mu^2y}$$
The condition of Lemma \ref{lemint} is fulfilled, so higher order perturbation analysis is necessary. We write
$$V(\varphi_\alpha(q_1,q_2))=\alpha(e_1+e_2)+ \alpha^{1/2}(ce_2^{1/2})+\alpha^\gamma(ae_1^\gamma+be_2^\gamma)+o(\alpha^\gamma)$$
with $\gamma \in ]0,1/2[$. We now compute in series at $\alpha=\infty$
$$\{\alpha H_2+\alpha^{1/2}(ce_2^{1/2})+\alpha^\gamma(ae_1^\gamma+be_2^\gamma),I_2+ \alpha^{-1/2} J+\alpha^{\gamma-1} F\} $$
The two first terms cancels, giving
$$\alpha^\gamma \{ae_1^\gamma+be_2^\gamma,I_2\}+\alpha^\gamma \{H_2,F\}+o(\alpha^\gamma).$$
Remark that the next term with $J$ can never mix with this one as it is of order $\alpha^0$, and we assumed $\gamma>0$. For integrability, this term has to be zero, so we can now use Lemma \ref{lemint}.
We evaluate $\{ae_1^\gamma+be_2^\gamma,I_2\}$ on the orbit \eqref{eqint1H2}, \eqref{eqint2H2}, and want to perform an integration in time staying in the same field. This is the same term as before, but now the condition $\gamma \in ]0,1/2[$ ensure that the case $\gamma=1/2$ can no longer occur, and thus the previous reasoning applies, giving that the integral cannot be in the same field as the integrand. This gives case $2$ of the Lemma.\\

For the third Hamiltonian $H_3$, the solutions are \eqref{eqint1H3}, \eqref{eqint2H3}. We now compute $\{ae_1^\gamma+be_2^\gamma,I_2\}$ and evaluate on this generic solution. Noting $x=\exp{i\lambda t},\;y=\exp{i\mu t}$, we obtain two large expressions for the coefficients in $a,b$ functions of $x,y$ respectively of the form
$$x^\gamma y^\gamma P_1(x,y)^{3\gamma-1}Q_1(x,y)^{-2\gamma-1}R_1(x,y),\qquad P_2(x,y)^{\gamma-1}Q_2(x,y)^{-2\gamma-1}R_2(x,y)$$
where $P_i,Q_i,R_i$ are polynomial and $P_i,Q_i$ irreducible non monomial. Moreover they are independent over $\mathbb{C}(x,y)$, so we can try to integrate them separately. For the coefficient of $b$, we use the weighted degrees
$$\hbox{deg}_1(x^iy^j)=i+j,\quad \hbox{deg}_2(x^iy^j)=i-2j$$
and corresponding valuations. The degrees and valuations of the coefficient are $\pm \gamma$, so Lemma \ref{lemint2} applies. We find that the valuations and degrees of $G$ should be $0$. This implies that $G$ is a constant. After differentiation, we find a constant $G$ cannot give the integral.

For the coefficient of $a$, the valuation/degree in $x$ is $[\gamma,-\gamma]$, in $y$ is $[\gamma,-\gamma]$, so Lemma \ref{lemint2} applies. When $\gamma\neq 1/3,2/3$, the valuation/degree in $x$ and $y$ of $G$ should be $[0,0]$ and $[0,0]$. Test of a constant $G$ does not give the integral.
When $\gamma\in\{1/3,2/3\}$, the term $P_1(x,y)^{3\gamma-1}$ becomes polynomial and thus the formula for the valuation/degree of $G$ changes. For $\gamma=1/3$ we find for $G$ $[0,2],[0,2]$, and for $\gamma=2/3$ $[0,4],[0,4]$. These two special cases are tested using linear algebra, and we find no solution.

\end{proof}

The case $H_3$ with $\gamma=1/3,2/3$ seems very particular in comparison to the others. In dimension $\leq 3$, it can only be realized by potential \eqref{eqpotnonint}, and this suggest special properties. Indeed, the divisors when performing the time integration are simpler, just one curve. This suggests that maybe a Darbouxian first integral of \eqref{eqpotnonint} could exist. 

\subsection{Proof of Theorem \ref{thmmain}}

\begin{proof}

The first integrability condition of Theorem \ref{thmmain} comes from Lemma \ref{lemsummit}. Let us look at the second condition. If two summits belong to the same edge, then the potential consisting of only its terms with frequencies on this edge should be integrable. So two cases appears. Either the edge is not collinear with $0$, and then after rotation/dilatation, we can apply Lemma \ref{lem2}, and thus the angle between the summits should be $\{\pi/2,2\pi/3,3\pi/4,5\pi/6\}$. Either it is collinear, and then the angle is $0$ or $\pi$.

For the third condition, the case where the angle is $2\pi/3,3\pi/4,5\pi/6$ leads to the three last integrable cases of Lemma \ref{lem2}. Computing the length ratio of the frequency vectors (which is invariant by rotation/dilatation), we obtain the condition. The fourth and fifth conditions comes from Lemma \ref{leminterior}.
\end{proof}

\section{Real integrable potentials}

Let us now prove corollary \ref{cor1}.

\begin{proof}
The potential $V$ is assumed to be a real trigonometric polynomial potential. Thus the set $\mathcal{S}(V)$ (the support of $\hat{V}$) is invariant by central symmetry, and so is $\mathcal{C}(V)$. Let us now consider the set $\Sigma \subset \mathbb{S}_{n-1}$ the (hyper)-spherical projection of the summits of $\mathcal{C}(V)$ on the $n-1$-dimensional unit sphere. This projection is always well defined as $\mathcal{C}(V)$ always contains $0$ (except for the case $V=0$ which is clearly separable). The set $\Sigma$ is a finite set, and also invariant by central symmetry. Due to Theorem \ref{thmmain}, the spherical distance between two different points of $\Sigma$ corresponding to summits of the same edge can be $\{\pi/2,2\pi/3,3\pi/4,5\pi/6,\pi\}$. As the polyhedra $\mathcal{C}(V)$ is convex, this implies in particular that the distance between any two different points of $\Sigma$ is at least $\pi/2$.

Now let us consider a point of $\Sigma$, and up to rotation we can assume it is $(0,\dots,0,1)$. Now due to symmetry, we also know that $(0,\dots,0,-1)$ also belongs to $\Sigma$. So all the other points of $\Sigma$ have to lie at a distance at least $\pi/2$ away from these two points. This is very restrictive, as it implies they are in the equator of the $\mathbb{S}_{n-1}$. The equator of $\mathbb{S}_{n-1}$ is a $n-2$-dimensional sphere inside the plane $k_n=0$. All the points of the equator lie exactly at a distance $\pi/2$ of both points $(0,\dots,0,\pm 1)$. So the second integrability condition of Theorem \ref{thmmain} with respect to these two points will always be satisfied.

We now apply the argument recursively. If there is a point in $\Sigma$ outside $(0,\dots,0,\pm 1)$, it lies in the equator, and so up to rotation we can assume it to be $(0,\dots,0,1,0)$. The point $(0,\dots,0,-1,0)$ will also be in $\Sigma$. Thus any additional point have to be in the equator of $\mathbb{S}_{n-2}$. Applying this argument down to $\mathbb{S}_1$, we obtain that up to rotation, the set $\Sigma$ satisfies
$$\Sigma \subset \{(0,\dots,0,\pm 1,0,\dots, 0)\in\mathbb{S}_{n-1}\}$$
This implies that $\mathcal{S}(V)$ is included in the coordinates axes, and thus the potential $V$ can be written
$$V(q)=\sum\limits_{i=1}^n V_i(q_i)$$
\end{proof}

\section{Polygonal tessellations}

If the hypothesis $V$ real is removed, we do not have the symmetry on the set $\Sigma$, and thus the recursive argument does not work any more at all. Indeed, the other point need not to be in the equator. However, the distance condition applies to every pair of points of $\Sigma$ corresponding to summits of the same edge. So we need to build tessellation of the $n-1$-dimensional sphere $\mathbb{S}_{n-1}$ with spherical polyhedrons with edge length $\pi/2,2\pi/3,3\pi/4,5\pi/6,\pi$. This problem becomes combinatorial, and thus we restrict ourselves to dimension $2$ and $3$. Remark that also contrary to the real case, the origin $0$ could possibly be outside $\mathcal{C}(V)$, and the spherical projection would not produce a proper tessellation of $\mathbb{S}_{n-1}$. We will prove that these cases lead in fact to subcases of proper tessellations. 

\subsection{In dimension $2$}

In dimension $2$, the tessellation problem is simply to recover the unit circle with arcs of length $\pi/2,2\pi/3,3\pi/4,5\pi/6,\pi$. The norm of the summit vectors $v_i$ can then sometimes be recovered using the third condition of Theorem \ref{thmmain}.

\begin{prop}\label{propdim2}
If $V$ is integrable, then the set of summits of $\mathcal{C}(V)$ is a subset one of the following sets up to rotation/symmetry
\begin{itemize}
\item With arcs $\pi,\pi$
$$v_1=(a,0),\;v_2=(-b,0),\quad a,b>0$$
\item With arcs $\pi,\pi/2,\pi/2$
$$v_1=(a,0),\;v_2=(-b,0),\;v_3=(c,0),\quad a,b,c>0$$
\item With arcs $5\pi/6,2\pi/3,\pi/2$
$$v_1=(a,0),\;v_2=(-a/2,a\sqrt{3}/2),\;v_3=(3a/2,-a\sqrt{3}/2),\quad a>0$$
$$v_1=(a,0),\;v_2=(-a/2,a\sqrt{3}/2),\;v_3=(a/2,-a/(2\sqrt{3})),\quad a>0$$
\item With arcs $3\pi/4,3\pi/4,\pi/2$
$$v_1=(a,0),\;v_2=(-a,a),\; v_3=(0,-a),\quad a>0$$
$$v_1=(a,0),\;v_2=(-a,a),\; v_3=(0,-2a),\quad a>0$$
\item With arcs $2\pi/3,2\pi/3,2\pi/3$
$$v_1=(a,0),\;v_2=(-a/2,a\sqrt{3}/2),\;v_3=(-a/2,-a\sqrt{3}/2),\quad a>0$$
\item With arcs $\pi/2,\pi/2,\pi/2,\pi/2$
$$v_1=(a,0),\;v_2=(-b,0),\;v_3=(c,0),\;v_4=(0,-d),\quad a,b,c,d>0$$
\end{itemize}
\end{prop}

The cases $1,2,6$ are separable cases. In cases $3,4,5$, we only have one free parameter. Up to dilatation of the coordinates, we can assume $a=2$. These cases lead respectively to the potentials in dimension $2$ number $(4,5),(2,3),1$ of Corollary \ref{cor2}.

\begin{lem}\label{lem1}
Any partial covering of length $<\pi$ of the circle with adjacent arcs of length $\pi/2,2\pi/3,3\pi/4,5\pi/6$ can be completed in a full covering given by Proposition \ref{propdim2}.
\end{lem}

\begin{proof}
As the smallest arc is of length $\pi/2$, there can be at most one arc in the partial covering. So the only thing to check is that all arc lengths $\pi/2,2\pi/3,3\pi/4,5\pi/6$ effectively appear in Proposition \ref{propdim2}, which is the case.
\end{proof}

\begin{proof}
Let us first remark that if $0$ is not inside $\mathcal{C}(V)$, then several (at most $2$ layers) edges can be projected on the same place of the circle. In this case, we consider only the edge which is the farthest from $0$. Using this procedure, we obtain a partial covering of the circle instead of a full covering. The arc length condition still applies, and that fact that $0$ is outside of $\mathcal{C}(V)$ requires that the total covering length is $<\pi$. Using Lemma \ref{lem1}, these partial covering can always be completed into full covering of the Proposition. So from now on, we can assume $0$ inside $\mathcal{C}(V)$ and a full covering of the circle.

As the smallest arc is of length $\pi/2$, there can be at most $4$ arcs. With $5$ arc length at our disposal, we just have to check $5^4$ possibilities. The possible arc length found are those of the Proposition. Remark that except in the case $\pi/2,\pi/2,\pi/2,\pi/2$, there are at most $3$ arcs. Thus up to rotation/symmetry, the order does not count (and for $\pi/2,\pi/2,\pi/2,\pi/2$ all the arcs are the same). We then fix a summit to $(a,0)$ using possibly a rotation. The norms of the other vectors $v_i$ are then recovered using the third condition of Theorem \ref{thmmain} when the arc is of length $2\pi/3,3\pi/4,5\pi/6$. For $3\pi/4,5\pi/6$, there are two possible ratios, and thus several possibilities for the lengths of the other $v_i$.
\end{proof}

\subsection{In dimension $3$}

\begin{defi}
We say that a tessellation of the sphere $\mathbb{S}_2$ is admissible if all the edges have length $\in \{\pi/2,2\pi/3,3\pi/4,5\pi/6,\pi\}$. The tessellation is said to be maximal if we cannot make a partition of a face adding only edge lengths $\pi/2,\pi$.
\end{defi}

Admissible tessellations are the spherical projection of convex polytopes satisfying the second condition of Theorem \ref{thmmain}. The maximal partition definition comes from an observation in dimension $2$. Several circle covering could be in fact repartitioned with smaller arc length. However, if this partition with smaller arc length only add lengths $\pi/2,\pi$, then no additional conditions on summit vectors length will be obtained (as the third condition of Theorem \ref{thmmain} then does not apply). So the final possible supports $\mathcal{S}(V)$ will contain the possible support for such non maximal tessellations. So in the following we will only have to study maximal tessellations.

\begin{prop}\label{propdim3}
The maximal admissible tessellations of the sphere $\mathbb{S}_2$ are made with the pieces
$$P_1:\left(\frac{\pi}{2},\frac{\pi}{2},\frac{\pi}{2}\right),\;P_2:\left(\frac{\pi}{2},\frac{\pi}{2},\frac{2\pi}{3}\right),\;P_3:\left(\frac{\pi}{2},\frac{\pi}{2},\frac{3\pi}{4}\right)$$
$$P_4:\left(\frac{\pi}{2},\frac{\pi}{2},\frac{5\pi}{6}\right),\;P_5:\left(\frac{\pi}{2},\frac{2\pi}{3},\frac{2\pi}{3}\right),\;P_6:\left(\frac{\pi}{2},\frac{2\pi}{3},\frac{3\pi}{4}\right)$$
and are the following
$$[P_1,P_1,P_1,P_1,P_1,P_1,P_1,P_1],[P_2,P_2,P_2,P_2,P_2,P_2],[P_1,P_1,P_3,P_3,P_3,P_3]$$
$$[P_1,P_1,P_2,P_2,P_4,P_4],[P_5,P_5,P_5,P_5],[P_1,P_5,P_6,P_6],[P_3,P_3,P_6,P_6]$$
\end{prop}

The first $4$ possible tessellations lead to separable (first case) or partially separable potentials. The $3$ last cases lead to the non separable potentials of Corollary \ref{cor2}. The tessellation $[P_5,P_5,P_5,P_5]$ leads to $\mathcal{C}$ tetrahedron with all the $v_i$ with the same norm (due to the distance $2\pi/3$ in piece $P_5$, giving the first non partially separable potential of Corollary \ref{cor2}. The tessellation $[P_1,P_5,P_6,P_6]$ gives two possible tetrahedrons for $\mathcal{C}$, but as the angle $3\pi/4$ appears, one of them is a subcase of the other, leading the second non partially separable potential of Corollary \ref{cor2}. The tessellation $[P_3,P_3,P_6,P_6]$ gives four possible tetrahedrons, but the angle $3\pi/4$ appears two times, and so $3$ tetrahedrons are in fact subcases of the largest tetrahedron, giving the last non partially separable potentials of Corollary \ref{cor2}.

\begin{proof}
Let us first find all possible pieces. The pieces are convex polygons on the sphere, and thus their perimeter should be at most $2\pi$. Moreover, those with perimeter $2\pi$ are hemispherical piece, and their edge form a large circle. Such a piece can therefore be subdivided by adding an point in the middle of the hemisphere. This point is at distance $\pi/2$ of all the others, and thus using such a piece would not lead to a maximal tessellation. Thus we can assume the perimeter to be $<2\pi$.

This implies the pieces are spherical triangles, and the possible triangle are exactly those of the Proposition \ref{propdim3}. The surface of these pieces is at least $\pi/2$, and so at most $8$ pieces are needed. We now make the use of the computer to combine these pieces. There are $6^8$ combinations to try, and the following conditions are necessary for them forming a tessellation.
\begin{itemize}
\item The total surface should be $4\pi$.
\item The number of edge of length $\pi/2,2\pi/3,3\pi/4,5\pi/6$ should be even respectively
\end{itemize}
The following cases are found
$$[P_1,P_1,P_1,P_1,P_1,P_1,P_1,P_1],[P_2,P_2,P_2,P_2,P_2,P_2],[P_1,P_1,P_3,P_3,P_3,P_3],$$
$$[P_1,P_1,P_1,P_1,P_5,P_5],[P_1,P_1,P_1,P_3,P_3,P_5],[P_1,P_1,P_2,P_2,P_4,P_4],$$
$$[P_5,P_5,P_5,P_5],[P_1,P_5,P_6,P_6],[P_3,P_3,P_6,P_6]$$
The $4$th and $5$th case are not possible because it would be necessary to glue $P_5$ to itself (the two $2\pi/3$ edges) or glue two $P_5$ at two different edges simultaneously (this requires an angle $\pi$ or $2\pi$ at the summit, which is not the case). So only the tessellations of the Proposition \ref{propdim3} remain, which can be effectively realized.
\end{proof}

\section{Exhibition of first integrals of integrable cases}

The potentials in dimension $2$ and $3$ which have passed all integrability conditions of Theorem \ref{thmmain} are exactly those of corollary \ref{cor2}. We know that there cannot be any other integrable potentials, but we still need to prove that these potentials are indeed integrable to finish our classification. For this we apply direct method to search for polynomial first integrals. We will use sparsest possible lattice $\mathcal{L}$ containing the support of $\hat{V}$.

Now we first apply the variable change $q\rightarrow iq$ to remove all the $i$ in the formulas. We will search for first integrals in $\mathbb{C}[p,\exp(k.q)_{k\in\mathcal{L}}]$. We moreover remark that the Galois action on the exponentials is multiplicative, the non separable potentials in dimension $2$ have at most $3$ exponentials terms and the non partially separable potentials in dimension $3$ have at most $4$ exponentials terms. So up to multiplication of the potential $V$ by a constant and Galois multiplicative action on the exponentials, it is enough to study integrability of the following Hamiltonians
\begin{align*}
H_1   =& \frac12\left(p_1^2+p_2^2\right)+e^{2q_1}+e^{-q_1+\sqrt{3}q_2}+e^{-q_1-\sqrt{3}q_2}\\
H_2   =& \frac12\left(p_1^2+p_2^2\right)+e^{2q_1}+e^{2q_2}+e^{-2q_1-2q_2}+\alpha e^{-q_1-q_2}\\
H_3   =& \frac12\left(p_1^2+p_2^2\right)+e^{2q_1}+e^{2q_2}+e^{-q_1-q_2}+\alpha e^{q_1}+ \beta e^{q_2}\\
H_4   =& \frac12\left(p_1^2+p_2^2\right)+e^{2q_1}+e^{2\sqrt{3}q_2}+e^{-q_1-\sqrt{3}q_2}\\
H_5   =& \frac12\left(p_1^2+p_2^2\right)+e^{2\sqrt{3}q_1}+e^{2q_2}+e^{-\sqrt{3}q_1-3q_2}\\
H_6   =& \frac12\left(p_1^2+p_2^2+p_3^2\right)+e^{q_1+q_2}+e^{-q_2+q_3}+e^{-q_2-q_3}+e^{-q_1+q_2}\\
H_7   =& \frac12\left(p_1^2+p_2^2+p_3^2\right)+e^{2q_1}+e^{-q_1+q_2}+e^{-q_2-q_3}+e^{-q_2+q_3}+\alpha e^{q_1}\\
H_8=& \frac12\left(p_1^2+p_2^2+p_3^2\right)+e^{2q_1}+e^{-q_1+q_2}+e^{-q_2+q_3}+e^{-2q_3}+\alpha e^{q_1}+\beta e^{-q_3}
\end{align*}
In every case we find sufficiently many commuting first integrals to prove Liouville integrability. The Hamiltonians $H_1,\dots,H_5$ are known to be integrable \cite{5,18,19} and $H_6$ is a reduction by center of mass of a Toda potential with periodic boundary condition, already known to be integrable \cite{1}. The $3$ dimensional Hamiltonians $H_6,\dots,H_8$ are more complicated, so we write down their first integrals $H_i,I_i,J_i$ which commute pairwise and are independent, proving they are integrable.\\

\begin{footnotesize}
\begin{equation*}\begin{split}
& I_6=p_1p_2p_3-p_3\left(e^{-q_1+q_2}-e^{q_1+q_2}\right)-p_1\left(e^{-q_2+q_3}-e^{-q_2-q_3}\right)\\
& J_6=p_1^4+p_2^4+p_3^4+4p_3^2\left(e^{-q_2-q_3}+e^{-q_2+q_3}\right)-4p_2p_3\left(-e^{-q_2+q_3}+e^{-q_2-q_3}\right)+\\
& 4p_2^2\left(e^{-q_2-q_3}+e^{-q_1+q_2}+e^{q_1+q_2}+e^{-q_2+q_3}\right)+4p_1p_2\left(e^{-q_1+q_2}-e^{q_1+q_2}\right)+\\
& 4p_1^2\left(e^{-q_1+q_2}+e^{q_1+q_2}\right)+4e^{q_1+q_3}+4e^{q_1-q_3}+4e^{-q_1+q_3}+4e^{-q_1-q_3}+2e^{2q_1+2q_2}+\\
& 2e^{-2q_1+2q_2}+2e^{-2q_2-2q_3}+2e^{-2q_2+2q_3}+12e^{2q_2}+12e^{-2q_2}\\
& I_7=p_1^4+p_2^4+p_3^4+\\
& \left(4\alpha e^{iq_1}+4e^{2q_1}+4e^{-q_1+q_2}\right)p_1^2+4p_1p_2e^{-q_1+q_2}+\left(4e^{-q_1+q_2}+4e^{-q_2+q_3}+4e^{-q_2-q_3}\right)p_2^2+\\
& \left(4e^{-q_2+q_3}-4e^{-q_2-q_3}\right)p_3p_2+\left(4e^{-q_2+q_3}+4e^{-q_2-q_3}\right)p_3^2+\\
& 4\alpha^2e^{2q_1}+4\alpha e^{q_2}+8\alpha e^{3q_1}+2e^{-2q_1+2q_2}+2e^{-2q_2-2q_3}+\\
& 2e^{-2q_2+2q_3}+8e^{q_1+q_2}+4e^{-q_1+q_3}+4e^{-q_1-q_3}+12e^{-2q_2}+4e^{4q_1}\\
& J_7=p_1^2p_2^2p_3^2+\\
& p_1^2p_2p_3\left(2e^{-q_2-q_3}-2e^{-q_2+q_3}\right)-2p_1p_2p_3^2e^{-q_1+q_2}+p_2^2p_3^2\left(2\alpha e^{q_1}+2e^{2q_1}\right)+\\
& p_1^2\left(e^{-2q_2+2q_3}+e^{-2q_2-2q_3}-2e^{-2q_2}\right)+p_1p_3\left(2e^{-q_1+q_3}-2e^{-q_1-q_3}\right)+\\
& p_2p_3\left(4\alpha e^{q_1-q_2-q_3}+4e^{2q_1-q_2-q_3}-4\alpha e^{q_1-q_2+q_3}-4e^{2q_1-q_2+q_3}\right)+\\
& p_3^2\left(e^{-2q_1+2q_2}+2\alpha e^{q_2}\right)+2\alpha e^{q_3}+2\alpha e^{-q_3}+2\alpha e^{q_1-2q_2+2q_3}+\\
& 2\alpha e^{q_1-2q_2-2q_3}-4\alpha e^{q_1-2q_2}-4 e^{2q_1-2q_2}+2e^{2q_1-2q_2+2q_3}+2e^{2q_1-2q_2-2q_3}\\
& I_8=p_1^4+p_2^4+p_3^4+\\
& p_1^2\left(4\alpha e^{q_1}+4e^{2q_1}-2\beta^2+4e^{q_2-q_1}\right)+4p_1p_2e^{q_2-q_1}+p_2^2\left(4e^{q_2-q_1}-2\beta^2+4e^{q_3-q_2}\right)+\\
& 4p_2p_3e^{-q_2+q_3}+p_3^2\left(4e^{-2q_3}-2\beta^2+4e^{-q_2+q_3}+4\beta e^{-q_3}\right)-4\alpha\beta^2e^{q_1}-4\beta^3e^{-q_3}+\\
& 4\alpha^2e^{2q_1}-4e^{2q_1}\beta^2-4\beta^2e^{-q_1+q_2}-4\beta^2e^{-q_2+q_3}+4\alpha e^{q_2}+8\beta e^{-3q_3}+8\alpha e^{3q_1}+\\
& 4\beta e^{-q_2}+4e^{4q_1}+4e^{-4q_3}+2e^{-2q_1+2q_2}+2e^{-2q_2+2q_3}+8e^{q_1+q_2}+8e^{-q_2-q_3}+4e^{-q_1+q_3}\\
& J_8=p_1^2p_2^2p_3^2+\\
& \left(2\beta e^{-q_3}+2e^{-2q_3}\right)p_2^2p_1^2-2p_2p_3p_1^2e^{-q_2+q_3}-2p_1p_2p_3^2e^{-q_1+q_2}+\left(2\alpha e^{q_1}+2e^{2q_1}\right)p_3^2p_2^2+\\
& \left(2\beta e^{-q_2}+e^{-2q_2+2q_3}\right)p_1^2-\left(4e^{-q_1+q_2-q_3}\beta+4e^{-q_1+q_2-2q_3}\right)p_2p_1+2e^{-q_1+q_3}p_3p_1+\\
& \left(4e^{q_1-q_3}\alpha\beta+4e^{2q_1-q_3}\beta+4e^{q_1-2q_3}\alpha+4e^{2q_1-2q_3}\right)p_2^2-\\
& \left(4e^{q_1-q_2+q_3}\alpha+4e^{2q_1-q_2+q_3}\right)p_3p_2+\left(2\alpha e^{q_2}+e^{-2q_1+2q_2}\right)p_3^2+\\
& 4\alpha\beta e^{q_1-q_2}+4\alpha\beta e^{q_2-q_3}+2\alpha e^{q_3}+2\beta e^{-q_1}+4\alpha e^{q_2-2q_3}+\\
& 4\beta e^{2q_1-q_2}+2\beta e^{-2q_1+2q_2-q_3}+2\alpha e^{q_1-2q_2+2q_3}+2e^{2q_1-2q_2+2q_3}+2e^{-2q_1+2q_2-2q_3}
\end{split}\end{equation*}
\end{footnotesize}

\section{Conclusion}

The necessary conditions for integrability given by Theorem \ref{thmmain} are sufficient in dimension $2,3$ and we can conjecture they are sufficient in any dimension.

The additional first integral are typically of high degree in the momenta, and it seems their degree is related to the ``complexity'' of the tiling. In particular, the length $5\pi/6$ seems to lead to higher degree first integrals. The tiling of the (hyper)-sphere can also be associated to a group, and the degree of the first integrals could be related to the size of this group. Moreover the conditions are very similar to the ones in \cite{17}, and so a similar classification in arbitrary dimension could be possible, probably also related to Dynkin diagrams. The question about more complicated first integrals could also be considered, in particular the case of Darbouxian first integrals and the potential \eqref{eqpotnonint}. Indeed, the Lemma \ref{lemint2} can be generalized for elementary functions, allowing to build necessary condition for integrability in elementary terms.

The search of potentials having only one additional first integral seems also possible: the starting point of the proof would be then a face of $\mathcal{C}(V)$. The Lemma \ref{lemint} can be applied to a generic orbit of the limit system, but which would not be explicit, as this system would not be Liouville integrable, and so the integration procedure of Lemma \ref{lemint2} would not apply. Still probably the commutative vector fields and invariant manifolds could be send birationally to $\mathbb{P}^k$ and linear vector fields, the integration problem reducing then to an equation of the type $Dg=f$ where $D$ is a derivation. If $D$ has no first integrals, i.e. injective in $K$, a condition on divisors similar to Lemma \ref{lemint2} should follow. If it has first integrals, additional reductions are necessary, for example when the faces of $\mathcal{C}(V)$ are Liouville integrable. This however requires to explicitly integrate all Liouville integrable trigonometric polynomial potentials.

Complete rational integrability allows to explicitly integrate an integrable Hamiltonian system. Clearly not all integrable systems are completely rationally integrable, but the fact that the $3$ Hamiltonians of section 2.5 are suggests that all the potentials of Corollary \ref{cor2} also are, and any other integrable trigonometric polynomial potential. Testing this still requires to rationally parametrize $n$ dimensional algebraic manifolds, a problem for which no systematic method is known.

\bibliography{Toda}


\end{document}